\documentclass[11pt,reqno]{amsart}
\usepackage{amsmath}
\usepackage{amssymb}
\usepackage{mathrsfs}
\usepackage{mhequ}
\usepackage{color}
\usepackage{graphicx}
\usepackage{enumerate}
\usepackage{todonotes}
\usepackage[unicode=true,
 bookmarks=true,bookmarksnumbered=false,bookmarksopen=false,
 breaklinks=false,pdfborder={0 0 1},backref=false,colorlinks=true]{hyperref}
\usepackage{bbm} 

\newtheorem{theorem}{Theorem}[section] 
\newtheorem{lemma}[theorem]{Lemma}     

\newtheorem{definition}[theorem]{Definition}

\newtheorem{remark}[theorem]{Remark}

\newcommand{\N} {\mathbb N}

\newcommand{\Z} {\mathbb Z}

\begin{document}

\title[On shadowing and Stochastic Stability]
 {On shadowing and Stochastic Stability}

\author[H. S. Puma]{Hector Suni Puma}
\address{Hector Suni Puma\\
Institute of Mathematics,  Department of Applied Mathematics\\
Universidade Estadual de Campinas\\
   13.083-859 Campinas - SP\\
   Brazil}
   \email{h212071@dac.unicamp.br}

\author[C.S. Rodrigues]{Christian S.~Rodrigues}
\address{Christian S.~Rodrigues\\
Institute of Mathematics,  Department of Applied Mathematics\\
Universidade Estadual de Campinas\\
   13.083-859 Campinas - SP\\
   Brazil and Max-Planck-Institute for Mathematics in the Sciences\\
  		Inselstr. 22\\
  		04103 Leipzig\\
  		Germany}
   \email{rodrigues@ime.unicamp.br}

\date{\today}

\begin{abstract}
  In this paper, we study stochastic stability of a dynamical system with shadowing property, which evolves under small random perturbation. We prove 
that time averages along the pseudo-trajectory converge with respect to stationary measure for the randomly perturbed dynamics.  In particular, we prove that stationary measures converge to physical measures when the noise level goes to zero if the dynamical system has a shadowing property.

\end{abstract}


\maketitle

\section{Introduction}

The study of statistical properties of Dynamical Systems is currently
present in almost all fields of science, from fundamental Mathematics
to applied modelling. In its essence, one wants to classify long term
behaviour of processes undergoing time transformation and understand how stable they are. This
task is particularly intricate in the case of randomly perturbed
dynamics, where several appearing phenomena have no counter part in
the deterministic case.

When the dynamical system evolves under small random 
perturbations, a way to characterise robust behaviour is by means of the so called \textit{stochastic stability}. We say that a system is stochastically stable if the stationary measure for the randomly perturbed dynamics is close, in some sense, to the invariant physical or Sinai-Ruelle-Bowen (SRB) measure, which in turn encodes important dynamical information. A more precise definition will be timely presented.

Since the 1980's, several results have been obtained under different
hypotheses leading to proving stochastic stability of some classes of
systems. To mention a few, Kifer has first proved stochastic stability
of diffeomorphisms having (uniformly) hyperbolic attractors and
expanding maps \cite{Kif86, Kif88}. Under slightly different technical
assumptions L.-S. Young then showed that the stability of (uniformly)
hyperbolic attractors of $C^2$- diffeomorphisms follows from the
persistence of hyperbolic structures \cite{You86}. Viana has shown how
to use spectral properties of transfer operators to infer stochastic
stability of several classes of maps \cite{Via97}. Later Ara\'{u}jo
has proved the existence and finiteness of physical measures for
randomly perturbed dynamics considering some special type of
perturbations \cite{Ara00, Ara01}. Then Alves and Ara\'{u}jo, based on
some construction already present in \cite{Ara00}, gave sufficient and
necessary conditions for stochastic stability of some non-uniformly
expanding maps \cite{AlA03}. More recently, Benedicks and Viana have
proved the stochastic stability of H\'{e}non like maps \cite{BeV06},
and Alves, Ara\'{u}jo, and V\'{a}squez the stochastic stability of a
type of non-uniformly hyperbolic diffeomorphisms with dominating
splitting \cite{AAV07}.

Another common way to characterise robustness of dynamics is through the idea of \textit{structural stability}. A dynamical system is said to be structurally stable if it is equivalent to some other dynamical system in its vicinity, choosing an appropriate topology. Among different consequences of structural stability, 
one may investigate the \textit{shadowing property}. Vaguely speaking, the shadowing property says that every pseudo-orbit, i.e. every sequence of points whose elements
are close enough to an orbit of a dynamical system, can be described
by a real orbit itself. Several authors have studied different types
of shadowing under different technical assumptions, mainly, because of
its close connection to structural stability \cite{Bow75, Kat80, Pil99, PiT10}. See \cite{Pil99} for a extensive historical account.

These two apparently distinct approaches to stability lead us to an important question: 
\textbf{how structural properties of dynamical systems are related to their statistical ones?}

The concept of shadowing is clearly related to stochastic stability because the orbits evolving under random perturbations are pseudo-orbits. In this context, proving stochastic stability equals to proving that most random orbits mimics the behaviour of real orbits (the unperturbed ones). The connection of shadowing with stochastic stability, however, has never been satisfactorily addressed. This problem has also been raised by Bonatti, D\'{i}az and Viana in~\cite[Problem D10]{BDV05}. To the best of our knowledge, the only attempt to study this connection was presented in the manuscript~\cite{Tod14}, where the author imposes several hypotheses and a much more restrictive context than ours.
 
In this paper, we study stochastic stability of a dynamical system with shadowing property, which evolves under small random perturbation. We define time averages with respect to continuous observables along the trajectory which converges to their averages with respect to an invariant measure. Using this framework we study time averages along the pseudo-trajectory that converge with respect to stationary measure for the randomly perturbed dynamics. Finally, we proof that stationary measures converge to physical measures when the noise level goes to zero if dynamical system has a shadowing property.

\section{Ergodicity and physical measures}\label{sec.ergo}

To get started, let $M$ be a compact smooth manifold of finite dimension with or without border. We endow $M$ with the distance $d$ induced by a Riemannian metric
and a normalised Riemannian volume form $m$, which we fix once and for all. When it is not stated otherwise, absolute continuity will be taken with respect to the probability $m$.

Given $\bold{x}=(x_j)_{j \in \Z_{0}^{+}}$, a sequence of points in $M$ from the non-negative integers. We define the following probability measure
\begin{equation}\label{eq.sum_seq}
 S_n(\bold{x})=\frac{1}{n+1}\sum_{j=0}^{n}\delta_{x_{j}},
\end{equation}
 where for every $y\in M$, the measure $\delta_{y}$ is the Dirac delta probability measure supported at $y$.
We shall consider the dynamics generated by the iteration of a continuous function $f:M \to M$, where we define $z_{j} = f^{j}(z_{0}) = \underbrace{f\circ \cdots \circ f}_\text{j times}(z_{0})$ for any initial condition $z_{0} \in M$. Hereafter we set $f^{0}(z) = z$, for every $z \in M$. To simplify our notation later on, we also define the following probability measure for each $z\in M$, which we denote
\begin{equation}\label{eq.cesar_orb}
S_n^f(z)=S_{n}\left(\{f^{j}(z)\}^{\infty}_{j=0}\right)=\frac{1}{n+1}\sum_{j=0}^{n}\delta_{f^{j}(z)}.
\end{equation}

As usual, we say that a probability measure $\mu$ is invariant by $f$ when $\mu\left(f^{-1}(A)\right) = \mu(A)$, for every Borel set $A \subset M$. Recall that, since $M$ endowed with the distance induced by the Riemannian metric is a compact metric space and $f$ is continuous, by the Krylov-Bogolyubov Theorem, there exists at least one $f$-invariant probability measure on $M$. So, if $\mu$ is such an $f$-invariant probability measure and $\varphi \in L^{1}(\mu)$ then, by the Birkhoff Ergodic Theorem, the limit
\begin{displaymath}
\tilde{\varphi} =  \lim_{n\to \infty}\frac{1}{n+1}\sum^n_{j=0}{\varphi\left(f^{j}(z)\right)}
\end{displaymath} 
exists for $\mu$-almost every point $z \in M$. Note that, on the one hand, for any $\varphi \in L^1(M)$, from (\ref{eq.cesar_orb}) we have
\begin{displaymath}
\begin{split}
\int{\varphi dS_n^f(z)}\\
&=\int\varphi d\left(\frac{1}{n+1}\sum_{j=0}^{n}\delta_{f^{j}(z)}\right)\\
&=\frac{1}{n+1}\sum_{j=0}^{n}\int\varphi d\delta_{f^{j}(z)}\\
&=\frac{1}{n+1}\sum_{j=0}^{n}\varphi \left(f^{j}(z)\right).
\end{split}
\end{displaymath}
Thus, there exists the limit
\begin{displaymath}
E_{z}(\varphi) \equiv \lim_{n\to \infty}\int{\varphi dS_n^f(z)}
\end{displaymath}
for almost every $z \in M$.
On the other hand, we also have that $\varphi \mapsto E(\varphi)=E_z(\varphi)$, for each $z\in B\subset M$, where $B$ is some set of positive measure, defines a non-negative linear functional on the space $L^1(M)$. Therefore, by the Riesz-Markov-Kakutani Representation Theorem, there exists a unique Radon measure $\nu$ on $M$, such that, 
\begin{equation}\label{eq.phis_measu}
\int{\varphi d\nu}=E(\varphi)=\lim_{n\to \infty}{\frac{1}{n+1}\sum_{j=0}^{n}\varphi \left(f^{j}(z)\right)}.
\end{equation}
Hence, a measure $\mu$ is ergodic if for $\mu$-almost every point $z\in M$, it holds that
$$S_n^f(z) \to \mu, \mbox{ as } n \to \infty$$
in the $weak^{\star}$-topology on the space of probability measures on $M$. These considerations lead us to the idea of physical measures.
\begin{definition}[Physical measure]
An invariant  probability measure $\mu$ on $M$ is called physical measure if equation (\ref{eq.phis_measu}) holds true for a set $\mathscr{B}(\mu) \subset M$, of positive Lebesgue measure, and for all continuous function $\varphi: M\to \mathbb{R}$. The set $\mathscr{B}(\mu) \subset M$ is called basin of the measure $\mu$.
\end{definition}


\section{Shadowing property}

Using the notation of Section~\ref{sec.ergo}, we shall consider the dynamics generated by the iteration of a continuous function $f:M \to M$. For a trajectory $z_{j} = f^{j}(z_{0})$ we can consider sequences of points which are close enough to this trajectory, what is formalised in the following definition.
\begin{definition}[Pseudo-trajectory] Given $\delta >0$, we say that a sequence of points $\bold{x} =(x_j)_{j \in \Z_{0}^{+}}$ in $M$ is a $\delta$-pseudo-trajectory of $f$ if
\begin{equation}\label{eq.pseudo_orbit}
d\left(f(x_{j}),x_{j+1}\right)\leq \delta , \text{ for all } j\geq 0.
\end{equation}
\end{definition}
We can now define the shadowing property we will be considering in this paper.

\begin{definition}[Shadowing]
We say that a dynamical system generated by $f$ has the shadowing property if, for every $\varepsilon > 0$, we can find $\delta > 0$, such that, for each $\delta$-pseudo-trajectory $\bold{x} =(x_j)_{j \in \Z_{0}^{+}}$ of $f$, there exists a point $z\in  M$, such that,
$$d(f^{j}(z),x_{j})<\varepsilon, \text{ for all } j\geq 0.$$
In this case, we say that $\bold{x} =(x_j)_{j \in \Z_{0}^{+}}$ is $\varepsilon$-shadowed by some $z \in M$.
\end{definition}\
\begin{definition}[Shadowable points]
Following~\cite{Mo16}, we also say that a point $y\in M$ is shadowable if for every $\varepsilon >0$ there is $\delta >0$, such that, every $\delta$-pseudo orbit $\bold{x}=(x_j)_{j \in \Z_{0}^{+}}$, with $x_0=y$ can be $\varepsilon$-shadowed.\\
\end{definition}\

We finish this section with a lemma adapted from~\cite{Mo16, Re17}, proving that every point in $M$ is shadowable.
\begin{lemma}\label{lem.shwpts}
Let $f: M\to M$ be a continuous function of a compact metric space $M$. Then, $f$  has the shadowing property if, and only if, every point in $M$ is shadowable.
\end{lemma}
\begin{proof}
Since $f$ has the shadowing property, for every $\varepsilon>0$, we can find $\delta >0$ such that, for every $\delta$-pseudo orbit  $\bold{x}=(x_j)_{j \in \Z_{0}^{+}}$ there exists a point  $z\in M$ with
\begin{displaymath}
d(f^j(z), x_j)\leq \varepsilon, \text{ for all } j\geq 0.
\end{displaymath}
 \\

Conversely, if every point in $M$ is shadowable, then for any $\varepsilon >0$, we have that for every point $w\in M$ there exists $\delta_{w} >0$, such that, every $\delta_{w}$-pseudo orbit, $\bold{x}=(x_j)_{j \in \Z_{0}^{+}}$ is $\varepsilon$-shadowed by some point in $M$. The set of open balls $\{B_{\delta_{w}}(w)\}_{w\in M}$ centred at each $w$ defines an open cover for $M$, and since $M$ is a compact metric space, we can extract a finite subcover $\{B_{\delta_{w_{i}}}(w_{i})\}_{i=1}^{n}$. Let $\delta=\displaystyle{\min_{i}}\{\delta_{w_i}\}$, since $\delta \leq \delta_{w_{i}}$ for all $i$, then  we consider any $\delta$-pseudo orbit $\bold{y}=(y_j)_{j \in \Z_{0}^{+}}$, of $f$ for some $i$. Thus $\bold{y}=(y_j)_{j \in \Z_{0}^{+}}$ is $\varepsilon$- shadowed by some point in $M$.\\
\end{proof}

\section{Random perturbations}

  Given any small enough $\varepsilon >0$, as before, we consider the dynamics generated by the iteration of a continuous function $f: M\to M$. A Markov chain is defined by a family $\{P_{\varepsilon}( \cdot | x): x\in M, \varepsilon > 0\}$ of Borel probability measures, such that, for each $x \in M$, we have $P_{\varepsilon}( \cdot | x)$ absolutely continuous with respect to the Lebesgue measure and $supp(P_{\varepsilon}(\cdot | x))\subset B_{\varepsilon}(f(x))$, where $B_{\varepsilon}(f(x))$ is the open ball of radius $\varepsilon$ centred in $f(x)$. Henceforth, the orbit of $f$ subject to small perturbation will be called random orbits. It is given by the sequence of random variables, $(x_j)_{j}$ for all $j=0, 1, 2,\dots $,  where each $x_{j+1}$ has probability distribution $P_{\varepsilon}( \cdot | x_j)$. This family of Markov chains will be called small random perturbation of $f$.  Under quite general assumptions, it is shown to exist a family of random maps representing the Markov chain~\cite{JKR15, JMPR19}.

Similarly to the importance of invariant measure for an application, the stationary measures for the Markov chain encodes important ergodic properties of the randomly perturbed dynamics. 
 \begin{definition}
A Borel probability measure $\mu_{\varepsilon}$ on $M$ is a stationary measure for the Markov chain if  
\begin{displaymath}
\mu_{\varepsilon}(G)=\int{P_{\varepsilon}(G | x)d\mu_{\varepsilon}}
\end{displaymath}
 for every $x\in M$ and  every Borel set $G\subset M.$ 
\end{definition}

We can define a shift map $F:M\times M^{\mathbb{N}}\to M\times M^{\mathbb{N}}$, $(x_{0}, \bold{x})\mapsto (x_1,\sigma(\bold{x}))$, where $\sigma$ is the left shift on the space of random orbit, such that, $\sigma(\bold{x})=(x_1, x_2, . . .)$ for any $\bold{x}=(x_{0}, x_{1}, x_{2}, \dots)$ and $x_{0}$. The stationary measure defined above is equivalent to the skew-product measure $\mu_{\varepsilon}\times {P_{\varepsilon}}^{\mathbb{N}}$, given by 
\begin{displaymath}
d(\mu_{\varepsilon}\times {P_{\varepsilon}}^{\mathbb{N}})(x_0, x_1, ..., x_j,...)=\mu_{\varepsilon}(x_0)P_{\varepsilon}(dx_1|x_0)...P_{\varepsilon}(dx_j|x_{j-1}) \dots,
\end{displaymath}
to be invariant under the shift map. Since $M$ is a compact metric space, by Tychonoff Theorem, the space $M\times M^{\mathbb{N}}$ is also compact. Moreover, F is a continuous map, then by the Krylov-Bogolyubov Theorem, there exists at least one $F$-invariant probability measure on $M\times M^{\mathbb{N}}$. So, if $\mu_{\varepsilon}\times {P_{\varepsilon}}^{\mathbb{N}}$ is such an $F$-invariant probability measure and $\psi$ is an integrable function, then, by the Birkhoff Ergodic Theorem, 
\begin{displaymath}
\widetilde{\psi}(\bold{x})=\lim_{n\to\infty}\frac{1}{n+1}\sum_{j=0}^{n}\psi(x_{j})=\lim_{n\to\infty}\frac{1}{n+1}\sum_{j=0}^{n}(\psi\circ \pi_1)(F^j(\bold{x}))
\end{displaymath}
exists for $\mu_{\varepsilon}\times {P_{\varepsilon}}^{\mathbb{N}}$-a.e $(x_0,\bold{x})\in M\times M^{\mathbb{N}}$, where $\pi_1:M\times M^{\mathbb{N}}\to M$ is the projection on the first factor. Thus, there exists the limit
$$E_{x_0}(\psi):=\lim_{n\to\infty}\int{\psi dS_n(\bold{x})}$$
for $\mu_{\varepsilon}$-almost every $x_0\in M$. Moreover, we have that $\psi \mapsto E(\psi)$, for any $x_0\in B\subset M$, where $B$ is some set of positive measure, defines a non-negative linear functional on the space $L^1(M)$. Thus, by the Riesz-Markov Kakutani Representation Theorem, there exists a unique Radon measure 
$\nu_{\varepsilon}$ on $M$, such that,
\begin{displaymath}
\int{\psi d\nu_{\varepsilon}}=E(\psi)=\lim_{n\to\infty}\frac{1}{n+1}\sum_{j=0}^{n}\psi(x_{j})
\end{displaymath}
Hence, the measure $\mu_{\varepsilon}$ is ergodic if for $\mu_{\varepsilon}$-almost every point $x_{0} \in M$, it holds that 
\begin{equation}\label{eq.converge_random}
S_n(\bold{x}) \to \mu_{\varepsilon}, \text{ as } n \to \infty
\end{equation}
in the $weak^{\star}$-topology sense in the space of probability measures on $M$.

Based on these, we can define the stability against random perturbation as following.
 \begin{definition} The system $(f, \mu)$ is stochastically stable under the small perturbation scheme $\{P_{\varepsilon}(\cdot | x ): x\in M, \varepsilon >0\}$ if $\mu_{\varepsilon}$ converges, in the weak$^{\star}$ sense, to the physical measure $\mu$, when the noise level $\varepsilon$ goes to zero.
 \end{definition}
 
 We shall consider the space of continuous functions from $M$ to $\mathbb{R}$, denoted by $C(M)$, which we endow with the $sup$-norm. Furthermore, let $\mathcal{B}\subset C(M)$ be any set of bounded continuous functions on $M$. These functions will be called \textit{observables} from $\mathcal{B}$. Then, we can prove the following easy lemmata for later use.

\begin{lemma}\label{Lm1} The measure $\mu$ is ergodic if for $\mu$-a.e point $z \in M$, it holds that  $S_n^f(z)\to \mu$, as $n \to \infty$. Then, for every $\varphi \in \mathcal{B}$, and for each $\varepsilon >0$ sufficiently small, there exists $n_0\in \mathbb{N}$, such that, 
\begin{displaymath}
\left|\int{\varphi dS_n^f(z)}-\int{\varphi d\mu}\right|\leq \varepsilon, \text{ for all } n\geq n_{0}.
\end{displaymath}
\end{lemma}

 \begin{proof}
 Suppose that $S_n^f(z)\to \mu$ in $M$ and  take $\varphi \in \mathcal{B}$. Given $\varepsilon>0$, consider the neighbourhood 
\begin{displaymath}
N_{\varepsilon, \varphi}(\mu)=\left\{\nu: \left|\int{\varphi d\nu}-\int{\varphi d\mu}\right|< \varepsilon\right\}.
\end{displaymath} 
By hypothesis, there exists $n_0\in \mathbb{N}$, such that, for all $ n\geq n_0$, we have $S_n^f(z)\in N_{\varepsilon, \varphi}(\mu)$. That is, $ \displaystyle{ \left| \int{\varphi dS_n^f(z)}-\int{\varphi d\mu}\right|< \varepsilon}$, for all $n\geq n_0$. 
 \end{proof}
 
\begin{lemma}\label{Lm2} The measure $\mu_{\varepsilon}$ is ergodic, if for $\mu_{\varepsilon}$-a.e point $x_{0} \in M$, it holds that 
 $S_n(\bold{x})\to \mu_{\varepsilon}$, as $n\to \infty$. Then, for every $\varphi \in \mathcal{B}$, and for each $\delta(\varepsilon)>0$, there exists $n_0\in \mathbb{N}$, such that, 
 $$\left|\int{\varphi dS_n(\bold{x})}-\int{\varphi d\mu_{\varepsilon}}\right|\leq \delta(\varepsilon), \forall n\geq n_0$$
 \end{lemma}
 \begin{proof}
 Suppose that $S_n(\bold{x})\to \mu_{\varepsilon}$ in $M$ and take $\varphi \in \mathcal{B}$. Given $\delta(\varepsilon)>0$, consider the neighborhood $\displaystyle{N_{\varepsilon, \varphi}(\mu_{\varepsilon})=\left\{\nu_{\varepsilon}: \left|\int{\varphi d\nu_{\varepsilon}}-\int{\varphi d\mu_{\varepsilon}}\right|< \delta(\varepsilon)\right\}}$. By hypothesis, there exists $n_0\in \mathbb{N}$, such that, for all $ n\geq n_0$ we have $S_n(\bold{x})\in N_{\varepsilon, \varphi}(\mu)$, that is $\displaystyle{\left|\int{\varphi dS_n(\bold{x})}-\int{\varphi d\mu_{\varepsilon}}\right|< \delta(\varepsilon)}$, for all $n\geq n_0$. 
 \end{proof}

\section{Shadowing and stochastic stability}
We can now state and prove the main result of this manuscript.

\begin{theorem}\label{thm.stochastic} 
Consider the dynamics generated by the map $f : M \to M$, and suppose that it has a unique absolutely continuous, invariant, ergodic probability $\mu$ and the shadowing property. Furthermore, suppose that under the perturbation scheme $\{P_{\varepsilon}(\cdot| x ): x\in M,\varepsilon >0\}$, with $P_{\varepsilon}(\cdot| x )$ absolutely continuous and $supp(P_{\varepsilon}(\cdot|x)) \subset B_{\varepsilon}(f(x))$ for every $x\in M$, there is a unique stationary probability measure $\mu_{\varepsilon}$ for every small $\varepsilon$. Then, the dynamics is stochastically stable.
\end{theorem}

\begin{proof}
\textbf{Step 1}:
Given any small $\varepsilon > 0$, note that the Markov chain defines,
for each $x_{0}\in M$, and for each $\delta(\varepsilon)>0$, a $\delta(\varepsilon)$-pseudo-trajectory of $f$. That is 
$$d(f(x_j), x_{j+1})\leq \delta(\varepsilon)\text{ for all } j\geq 0,$$ 
where each $x_{j+1}$ has probability distribution $P_{\varepsilon}( \cdot | x_j)$. We will denote by $B_{\delta, \varepsilon}$ a set of realisations of the Markov chain. By hypothesis, $f$ has the shadowing property, which means that, for every $\varepsilon > 0$, we can find a $\delta(\varepsilon)>0$, such that, for any $\delta(\varepsilon)$-pseudo-trajectory, $\bold{x}\in B_{\delta, \varepsilon}$, of $f$, there exist $z\in M$, such that, 
\begin{equation}\label{eq.1Thm}
d(f^j(z),x_j)<\varepsilon , \text{ for all } j\geq 0.
\end{equation}
Furthermore, since $M$ is compact, by Lemma~\ref{lem.shwpts}, every point in $M$ is shadowable. So, every $x \in M$ can be taken as $x = x_{0}$ for a random orbit $(x_{j})_{j}$, such that, there exist $z(x_{0}) \in M$ so that $d(f^j(z),x_j)<\varepsilon , \text{ for all } j\geq 0$.\\

\textbf{Step 2}:  We want to assure that the points $z$ can be taken $\mu$-typical. By hypothesis, $f$ has the shadowing property then, by Lemma~\ref{lem.shwpts}, every point in $M$ is shadowable. 
  Take any $x\in M$, such an arbitrary shadowable point, and small enough $\varepsilon > 0$, then there exists a $\delta(\varepsilon)>0$, such that, for each $\delta(\varepsilon)$-pseudo-trajectory, $\bold{x}\in B_{\delta, \varepsilon}$, of $f$ with $x_0=x$, there exist $z\in M$ such that,
 \begin{equation}\label{eq.1Thm}
d(f^j(z),x_j)<\varepsilon , \text{ for all } j\geq 0.
\end{equation}

Moreover, $\mu$ is ergodic if for $\mu$-a.e point $z\in M$, we have $S_n^f(z)\to \mu$, as $n\to\infty$. For such measure, we consider the following set of points
\begin{equation}
Z_{\mu}:=\left\{z\in M:\frac{1}{n+1}\sum^n_{j=0}{\varphi(f^j(z))}\underrightarrow{n \to \infty} \int{\varphi}d\mu\right\},
\end{equation}
defined for every continuous function $\varphi \in C(M)$. As the set $C(M)$ is separable, we can take a dense sequence $\left( \phi_{k}\right)_{k}$ in $C(M)$. By the Birkhoff Ergodic Theorem, for each $\phi_{k}$, there is a set $M_{k}$ with $\mu\left( M_{k} \right) = 1$, such that, for $z \in M_{k}$, it holds
\begin{displaymath}
\lim_{n \to \infty}\frac{1}{n+1}\sum^n_{j=0}{\phi_{k}(f^j(z))} = \int{\phi}_{k}d\mu.
\end{displaymath}
Defining $X = \bigcap_{k} M_{k}$ we have $\mu\left(X\right) = \mu \left( \bigcap_{k} M_{k} \right) = 1$, since the countable intersection of full measure sets is a full measure set. Furthermore, on $X$ we have for every $\phi_{k}$
\begin{displaymath}
\lim_{n \to \infty}\frac{1}{n+1}\sum^n_{j=0}{\phi_{k}(f^j(z))} = \int{\phi}_{k}d\mu.
\end{displaymath}
Let $\varphi \in C(M)$, then there is a subsequence $\left( \psi_{k} \right)_{k}$ of $\left( \phi_{k}\right)_{k}$ such that for every $\varepsilon >0$,  there is a $K \in \N$ so that for every $k > K$ we have $\lVert \varphi - \psi_{k} \rVert_{\infty} \leq \varepsilon$. Thus, for every $z \in X$ we have

\begin{eqnarray*}
    \left| \frac{1}{n+1}\sum_{j=0}^{n}\varphi\left(f^{j}(z)\right) \right. & - & \left.\int \varphi d\mu \right| \\ 
    &=&  \left| \frac{1}{n+1}\sum_{j=0}^{n}\varphi\left(f^{j}(z)\right) - \frac{1}{n+1}\sum_{j=0}^{n}\psi_{k}\left(f^{j}(z)\right)\right. \\ 
    &&   \left.+ \frac{1}{n+1}\sum_{j=0}^{n}\psi_{k}\left(f^{j}(z)\right) - 
\int \psi_{k}d\mu + \int (\psi_{k}-\varphi)d\mu \right|.
\end{eqnarray*}
Note that,
\begin{displaymath}
\left|\int (\psi_{k}-\varphi)d\mu \right| \leq \varepsilon \mu\left(M\right),
\end{displaymath}
\begin{displaymath}
 \left| \frac{1}{n+1}\sum_{j=0}^{n}\varphi\left(f^{j}(z)\right) - \frac{1}{n+1}\sum_{j=0}^{n}\psi_{k}\left(f^{j}(z)\right)\right|\leq \varepsilon,
\end{displaymath}
and 
\begin{displaymath}
\lim_{n \to \infty} \frac{1}{n+1}\sum_{j=0}^{n}\psi_{k}\left(f^{j}(z)\right) - 
\int \psi_{k}d\mu = 0.
\end{displaymath}
Then, by the triangle inequality we obtain,
\begin{displaymath}
 \limsup_{n} \left| \frac{1}{n+1}\sum_{j=0}^{n}\varphi\left(f^{j}(z)\right) - \int \varphi d\mu \right| \leq \varepsilon (1 + \mu\left(M\right) ),
\end{displaymath}
which is valid for every $\varepsilon$. Thus,
\begin{displaymath}
 \limsup_{n} \left| \frac{1}{n+1}\sum_{j=0}^{n}\varphi\left(f^{j}(z)\right) - \int \varphi d\mu \right| = 0.
\end{displaymath}
And hence,
\begin{displaymath}
 \lim_{n} \left| \frac{1}{n+1}\sum_{j=0}^{n}\varphi\left(f^{j}(z)\right) - \int \varphi d\mu \right| = 0.
\end{displaymath}
Therefore, we conclude that $\mu(Z_{\mu})=1$.

Note also that, since $\mu$ is absolutely continuous with respect to the Lebesgue measure $m$, having $\mu(Z_{\mu})=1$ implies $m(Z_{\mu}) > 0$, so the measure $\mu$ is physical. We do not require each random trajectory to be shadowed by a unique point $z$.
\\


\textbf{Step 3}:
Considering a given random trajectory, we can prove that for every fixed $\varepsilon >0$, an observable taken in the set of bounded continuous function $\varphi \in \mathcal{B}$, and for any sufficiently large $N\in \mathbb{N}$, there exists a constant $C(\varphi)$, such that
\begin{equation}\label{eq.thm}
\left|\int{\varphi dS_n^f(z)}-\int{\varphi dS_n(\bold{x})}\right|\leq C(\varphi)\varepsilon, \text{ for all }    n>N,
\end{equation}
holds.
Indeed,
\begin{align}\label{eq.ineq}
&\left| \int{\varphi dS_n^f(z)} - \int{\varphi dS_{n}(\bold{x})}\right| \nonumber\\ 
&\qquad \qquad =\left|\int{\varphi d\left(\frac{1}{n+1}\sum^n_{j=0}{\delta_{f^j(z)}}\right)}-\int{\varphi d\left(\frac{1}{n+1}\sum^n_{j=0}{\delta_{x_j}}\right)}\right|\nonumber \\
& \qquad \qquad \leq \frac{1}{n+1}\sum^n_{j=0}\left|\int{\varphi d\delta_{f^j(z)}}-\int{\varphi d\delta_{x_j}}\right|\nonumber  \\
& \qquad \qquad = \frac{1}{n+1}\sum^n_{j=0}\left|\varphi(f^j(z))-\varphi(x_j)\right|\nonumber \\
& \qquad \qquad \leq  \frac{1}{n+1}\sum^n_{j=0}\sup_{z\in M}\left|\varphi(f^j(z))-\varphi(x_j)\right|.  
\end{align}

Since $\varphi \in \mathcal{B}$, there exists a constant $C(\varphi)>0$, such that, 
\begin{displaymath}
\displaystyle{\sup_{z\in M}|\varphi(z)|}\leq C(\varphi).
\end{displaymath}
 Then, we can write from the equation (\ref{eq.ineq}) for every $j\geq 0$,
\begin{displaymath} 
 \displaystyle{\sup_{z\in M}\left|\varphi(f^j(z))-\varphi(x_j)\right|}\leq \sup_{z\in M}|\varphi(z)|d(f^j(z), x_j).
\end{displaymath} 
  Thus, we have 
\begin{displaymath}
\left|\int{\varphi dS_n^f(z)}-\int{\varphi dS_n(\bold{x})}\right|\leq C(\varphi)\varepsilon \text{ for all } n\geq N.
\end{displaymath}\\

\textbf{Step 4}:
We finally can show stochastic stability. Note that $|\int{\varphi d\mu}-\int{\varphi d\mu_{\varepsilon}}|$ can be rewritten in therms of $\int{\varphi dS_n^f(z)}$ and $\int{\varphi dS_n(\bold{x})}$ and, then applying triangle inequality, we have the following 
\begin{displaymath}
\begin{split}
\left|\int{\varphi d\mu}-\int{\varphi d\mu_{\varepsilon}}\right|
&\leq  \left|\int{\varphi d\mu}-\int{\varphi dS_n^f(z)}\right|+\left|\int{\varphi dS_n(\bold{x})}-\int{\varphi d\mu_{\varepsilon}}\right| \\ 
&+ \left|\int{\varphi dS_n^f(z)}-\int{\varphi dS_n(\bold{x})}\right| \\ 
&\leq  (C(\varphi)+1)\varepsilon+\delta(\varepsilon).
\end{split}
\end{displaymath}
Observe that the first term from the right of the inequality follows from Lemma~\ref{Lm1}, which by step 2, can be applied to almost all orbits of $f$ with $z$ chosen $\mu$ almost surely. The second term  follows from Lemma~\ref{Lm2}, which can be applied to almost all random trajectories. The last term is a consequence from (\ref{eq.thm}). Moreover, since $B_{\delta,\varepsilon}$ is a set of realisations of Markov chain and the measure $S_n(\bold{x})$, is defined as (\ref{eq.sum_seq}), then from (\ref{eq.converge_random}) and ergodicity, we have that in the limit 
  $\mu_{\varepsilon}(B_{\delta,\varepsilon})\to 1$, as $N\to \infty$.
This concludes the proof.
\end{proof}

\begin{remark}
Uniqueness of stationary measures holds assuming some transitivity on the dynamics, for example. The context in~\cite[Proposition 2.1]{Ara00} is what we have in mind. In the cases where $f$ has more than one physical measure, a more general notion of stochastic stability needs to be applied. The natural way to go is to consider that the simplex of stationary measures are weak* close to the simplex generated by the physical measures of the unperturbed system when $\varepsilon$ is small, in the lines of~\cite{Ara00}.
\end{remark}

\begin{remark}
The quantifier $\delta(\varepsilon)$ should be considered as a continuous function controlling the type of shadowing, $\delta: \mathbb{R}^+\to \mathbb{R}^+$, such that, $\delta(\varepsilon)\to 0$ as $\varepsilon \to 0$. It is the same appearing in (\ref{eq.pseudo_orbit}), but we introduce an explicit dependence on the size of the random perturbation. Therefore, the approach taken here allows for different versions of shadowing property by choosing different functions  $\delta(\varepsilon)$. This includes, for example, the Lipschitz shadowing~\cite{PiT10} as well as some other types of shadowing properties.
\end{remark}

It is possible to formulate a similar result to Theorem~\ref{thm.stochastic} in terms of random maps, given that there is an equivalent formulation of the perturbation scheme through iteration of random functions via the representation of Markov chains~\cite{JKR15, JMPR19}.

\section*{Acknowledgements}

This study was financed in part by the Coordena\c{c}\~ao de Aperfei\c{c}oamento de Pessoal de N\'ivel Superior - Brasil (CAPES) - Finance Code 001.

	\noindent C. S. R. would like to acknowledge support from the Max Planck Society, Germany, through the award of a Max Planck Partner Group for Geometry and Probability in Dynamical Systems. C. S. R. has also been supported by the Brazilian agency: grant \#2016/00332-1, grant \#2018/13481-0, grant \#2020/04426-6, S\~{a}o Paulo Research Foundation (FAPESP).
The opinions, hypotheses and conclusions or recommendations expressed in this work are the responsibility of the authors and do not necessarily reflect the views of FAPESP. 
\vspace{1cm}


\bibliographystyle{amsalpha}

\end{document}